\documentclass[english,10pt]{amsart}
\usepackage[T1]{fontenc}
\usepackage[utf8]{inputenc}
\usepackage[english]{babel}
\usepackage{graphicx,xcolor}

\usepackage{amsmath,amsfonts,enumerate,amssymb}
\usepackage{amsthm}
\usepackage{float}

\newtheorem{Def}{Definition}[section]
\newtheorem{Lem}{Lemma}[section]
\newtheorem{Teo}{Theorem}[section]
\newtheorem{Cor}{Corollary}[section]
\newtheorem{Rem}{Remark}[section]

\newtheorem{NotOurTheorem}{Theorem}

\renewcommand{\le}{\leqslant}
\renewcommand{\ge}{\geqslant}

\newcommand*{\xx}{\mathbf{x}}
\newcommand*{\yy}{\mathbf{y}}
\newcommand*{\ww}{\mathbf{w}}
\newcommand*{\ee}{\mathbf{e}}
\newcommand*{\rr}{\mathbf{r}}
\newcommand*{\cc}{\mathbf{c}}

\newcommand*{\RR}{\mathbb{R}}
\newcommand{\uRR}{\underline{\mathbb{R}}}
\newcommand*{\NN}{\mathbb{N}}

\newcommand*{\Simplex}[1]{\overline{S^{#1}}}
\newcommand*{\Sinf}{\Simplex{\infty}}

\newcommand*{\mol}{\overline{m}}
\newcommand*{\molsegm}[1]{\overline{m^{#1}}}
\newcommand*{\mul}{\underline{m}}
\newcommand*{\mulsegm}[1]{\underline{m^{#1}}}

\def\ds{\displaystyle}

\begin{document}

\title{Minimax and maximin problems for sums of translates on the real axis}

\author{Tatiana M. Nikiforova}

\date{}


\thanks{The work was performed as part of research conducted in the Ural Mathematical Center with the financial support of the Ministry of Science and Higher Education of the Russian Federation (Agreement number 075-02-2024-1377)}

\begin{abstract}
Sums of translates generalize logarithms of weighted algebraic polynomials. The paper presents the solution to the minimax and maximin problems on the real axis for sums of translates. We prove that there is a unique function that is extremal in both problems. The key in our proof is a reduction to the problem on a segment. For this, we work out an analogue of the Mhaskar-Rakhmanov-Saff theorem, too.    
\end{abstract}

\maketitle

\section{Introduction}

\subsection{The minimax and maximin problems on a segment}
In 1978, B.~D.~Bojanov generalized the famous Chebyshev's theorem on polynomials least deviating from zero on a segment \cite{Bojanov}.

\begin{NotOurTheorem} \label{theorem:Bojanov}
Let $n \in \NN$ and $\nu_1, \ldots, \nu_n > 0$ be integers. There is a unique set of points $x^*_1 \le \ldots \le x^*_n$
such that
\[
\|(x-x^*_1)^{\nu_1}\cdot \ldots \cdot (x-x^*_n)^{\nu_n}\| = \inf \limits_{x_1 \le \ldots \le x_n} \|(x-x_1)^{\nu_1}\cdot \ldots \cdot (x-x_n)^{\nu_n}\|,
\]
where $\|\cdot\|$ is the sup norm over $[0, 1]$. Moreover, $0 < x^*_1 < \ldots < x^*_n < 1$, and the extremal polynomial $T(x) := (x-x^*_1)^{\nu_1}\cdot \ldots \cdot (x-x^*_n)^{\nu_n}$ is characterized by an equioscillation property: there exists an array of points $0 = t_0 < t_1 < \ldots < t_{n-1} < t_n = 1$ such that
\[
T(t_k)=(-1)^{\nu_{k+1}+\ldots+\nu_{n}}\|T\|, \quad k=0,\ldots,n.
\]
\end{NotOurTheorem}
By taking logarithm, we can write Bojanov's problem in the following form
\[
\mathbf{minimize} \text{ (in } x_1 \le \ldots \le x_n) \quad
\max \limits_{x \in [0,1]} \sum \limits_{j=1}^n \nu_j \log |x-x_j|.
\]
It is natural to consider weighted sup norms. That is, for a weight function $w(x) \ge 0$, $x\in [0,1]$, we can also consider the problem
\[
\mathbf{minimize} \text{ (in } x_1 \le \ldots \le x_n) \quad
\max \limits_{x \in [0,1]} \left( \log w(x) + \sum \limits_{j=1}^n \nu_j \log |x-x_j| \right).
\]

In 2000, P.~C.~Fenton \cite{Fenton} considered a generalization of the weighted Chebyshev problem and solved a dual maximization problem, too. He worked with the so-called sums of translates functions, generalizing the above weighted sum of logarithms. 
Now, to formulate Fenton's result, let us give some definitions. We use a different notation than Fenton's original one. Instead, we use the notation from the recent papers \cite{FNR} and \cite{NewFNR}, to which we will return later. 

\begin{Def}
Let $0<p\le\infty.$
A function $K \colon (-p,0) \cup (0, p) \to \RR$ is called a kernel function if $K$ is concave on $(-p,0)$ and on $(0, p)$
and $\lim \limits_{t \downarrow 0} K(t)=\lim \limits_{t \uparrow 0} K(t),$ which are either real or equal to $-\infty$.
\end{Def} 

If additionally $K$ is decreasing on $(-p, 0)$ and increasing on $(0, p)$, then we call $K$ \textit{monotone}.

We extend $K$ by defining
\[
K(0) = \lim \limits_{t \to 0} K(t), \quad K(-p) = \lim \limits_{t \downarrow -p} K(t), \quad K(p) = \lim \limits_{t \uparrow p} K(t).
\]
If
$K(0) = - \infty$, the kernel function $K$ is called \textit{singular}.

\begin{Def}
Let $A$ be a segment, a semiaxis or $\RR$. We call a function $J \colon A \to \uRR := \RR \cup \{-\infty\}$ an external $n$-field function or simply a field on $A$ if $J$ is bounded above on $A$ and it assumes finite values at more than $n$ different points of $A,$ where in the case of a segment we count boundary points with weights $1/2.$
\end{Def}

Here we impose precisely the conditions on the weights of the points to keep consistency with the case of a segment. In this case, it is necessary that there are at least $n$ interior points and some additional one anywhere in the segment, where the field is finite.

Let $r_1, \ldots, r_n > 0$. Denote $\yy := (y_1, \ldots, y_n),$ where $y_1 \le \ldots \le y_n$. Replacing $\log |\cdot-y_j|$ by $K(\cdot-y_j)$ and $\log w$ by $J$, we obtain \textit{the sum of translates function}:
\begin{align} \label{eq:sumOfTranslates}
F(\yy,t) = J(t) + \sum \limits_{j=1}^n r_j K(t-y_j).
\end{align}
The sum of translates method originates from Fenton. Initially, Fenton's goal was to prove a conjecture of P.~D.~Barry from 1962 on the growth of entire functions. Fenton succeeded in this in 1981~\cite{FentonBarry}. And even though this conjecture was proved a little earlier by A.~A.~Goldberg~\cite{Goldberg}, Fenton got other nice results in the theory of entire functions using his approach of the sums of translates \cite{FentonOther1}, \cite{FentonOther2}.

Consider a segment $[a,b]$. In what follows, we will denote by $\Simplex{[a,b]}$ the closed simplex
\[
\Simplex{[a,b]} := 
\{\yy = (y_1, \ldots, y_n) \in \RR^n: \ a \le y_1 \le \ldots \le y_n \le b\}.
\]

Let $\yy \in \Simplex{[a,b]}$ and $F(\yy,t)$ be defined for $t \in [a,b]$. Denote
\begin{align*}
m^{[a,b]}_0(\yy) & := \sup \limits_{t \in [a, y_1]} F(\yy, t), \quad
m^{[a,b]}_n(\yy) := \sup \limits_{t \in [y_n, b]} F(\yy, t),\\
m^{[a,b]}_j(\yy) & := \sup \limits_{t \in [y_j, y_{j+1}]} F(\yy, t), \quad j = 1, \ldots, n-1,
\end{align*}
and 
\[
\molsegm{[a,b]}(\yy) := \max \limits_{j=0,\ldots,n} m^{[a,b]}_j(\yy) = \sup \limits_{t \in [a,b]} F(\yy, t), \quad
\mulsegm{[a,b]}(\yy) := \min \limits_{j=0,\ldots,n} m^{[a,b]}_j(\yy).
\]
Note that for any $\yy \in \Simplex{[a,b]}$ the value of $\molsegm{[a,b]}(\yy)$ is finite. Indeed, we have $\molsegm{[a,b]}(\yy) > -\infty$, since $J$ is finite at least at $n+1$ points.
Further, $\molsegm{[a,b]}(\yy) < \infty,$ as $K$ is concave and $J$ is bounded above.

Next, we also define
\[
M(\Simplex{[a,b]}) := \inf \limits_{\yy \in \Simplex{[a,b]}} \left(\molsegm{[a,b]}(\yy)\right), \quad m(\Simplex{[a,b]}) := \sup \limits_{\yy \in \Simplex{[a,b]}} \left(\mulsegm{[a,b]}(\yy)\right).
\]

If $a=-b,$ then we will write $\Simplex{b}, \ m^{b}_j(\yy), \ \molsegm{b}(\yy), \ \mulsegm{b}(\yy), \ M(\Simplex{b})$ and $m(\Simplex{b}).$

A point $\yy \in \Simplex{[a,b]}$ is called an equioscillation point on $[a,b]$ if
\[
m^{[a,b]}_0(\yy) = m^{[a,b]}_1(\yy) = \ldots = m^{[a,b]}_n(\yy).
\]

Now consider $[a,b] = [0,1]$. Let $K^1$ be a kernel defined on $[-1,1]$ and $J^1$ be an $n$-field defined on $[0,1].$ Fenton proved the following theorem \cite{Fenton}.

\begin{NotOurTheorem} \label{theorem:Fenton}
Let $K^1$ be a monotone kernel defined on $[-1,1]$, $K^1$ be twice differentiable on $[-1,1] \setminus \{0\}$, $(K^1)'' < 0$ and
\[
\lim \limits_{t \to 0} |(K^1)'(t)| = \infty.
\]
Let $n \in \NN, \ r_1=\ldots=r_n=1$. 
Assume that $J$ is a concave function on $(0,1)$ and put
\[
J(0) = \lim \limits_{t \downarrow 0} J(t), \
J(1) = \lim \limits_{t \uparrow 1} J(t).
\]
Consider the sum of translates function \eqref{eq:sumOfTranslates}. 
Then there is a unique point $\ww = (w_1, \ldots, w_n) \in \Simplex{[0,1]}$ such that
\[
\molsegm{[0,1]}(\ww) = M(\Simplex{[0,1]}), \quad
\mulsegm{[0,1]}(\ww) = m(\Simplex{[0,1]}).
\]
This extremal point $\ww$ is the unique equioscillation point on $[0,1].$
\end{NotOurTheorem}

The problem of minimizing $\molsegm{[0,1]}$ is called the minimax problem, and the problem of maximizing $\mulsegm{[0,1]}$ is called the maximin problem.

At present, B.~Farkas, B.~Nagy, and Sz.~Gy.~R\'{e}v\'{e}sz have fruitfully developed the subject of sums of translates in their research. The following theorem combines the results of their articles \cite{FNR} and \cite{NewFNR}.
\begin{NotOurTheorem} \label{theorem:FNR}
Let $K^1: (-1, 0) \cup (0, 1) \to \RR$ be a monotone kernel function. Let $n \in \NN, \ r_1, \ldots, r_n > 0$ be arbitrary.  Assume that $J^1$ is an $n$-field function on $[0,1]$. Consider the sum of translates function \eqref{eq:sumOfTranslates}. 
Then 
\[
M(\Simplex{[0,1]}) = m(\Simplex{[0,1]})
\]
and 
there exists some point $\ww = (w_1, \ldots, w_n) \in \Simplex{[0,1]}$ at which the simplex minimax is attained: 
\[
\mol(\ww) = M(\Simplex{[0,1]}). 
\]
Furthermore, there are no $\xx, \yy$ with finite local maxima $m^{[0,1]}_j$ such that
\begin{gather} \label{theorem:FNR:weakIntertwining}
m^{[0,1]}_j(\xx) > m^{[0,1]}_j(\yy), \quad j=0,\ldots,n.
\end{gather}
If there exists an equioscillation point $\ee$, then it is the minimax and maximin point, i.e.,
\[
M(\Simplex{[0,1]})=\mol(\ee)=\mul(\ee)=m(\Simplex{[0,1]}). 
\]
Additionally,
\begin{enumerate}
\item If $J^1$ is upper semicontinuous or $K^1$ is singular, then there exists an equioscillation point.
\item If $K^1$ is singular and strictly concave, then the equioscillation point is unique. Moreover, the so-called intertwining property holds: the strict inequality \eqref{theorem:FNR:weakIntertwining} can be replaced by the non-strict one, if $\xx \neq \yy.$ In particular, this implies that the equioscillation point is the unique minimax and maximin point.
\end{enumerate}
\end{NotOurTheorem}

In conclusion, it is worth noting that the sums of translates approach has found its place not only in problems on a segment, but also on a torus. It all started in 2013 with a conjecture of G.~Ambrus, K.~M.~Ball and T.~Erd\'{e}lyi \cite{AmbrusBallErdelyi} that for any $2\pi$-periodic, even and convex on $(0, 2\pi)$ function~$f$ the expression
\[
\min \limits_{\theta \in [0, 2\pi)} \sum \limits_{j=1}^n f(\theta-\theta_j)
\]
is maximized when the nodes $\theta_1, \ldots, \theta_n$ are uniformly distributed on $[0, 2\pi)$ (it follows that the local maxima of $\sum \limits_{j=1}^n f(\theta-\theta_j)$ are equal). This conjecture was proved by D.~P.~Hardin, A.~P.~Kendall and E.~B.~Saff \cite{HardinKendallSaff} in the same year. In 2018, Farkas, Nagy and R\'{e}v\'{e}sz presented a solution of the minimax problem for functions $\ds F(\yy,t) = K_0(t) + \sum \limits_{j=1}^n K_j(t-y_j)$ \cite{TorusFNR}. They assumed that $K_0, \ldots, K_n: \ \RR \to [-\infty, 0)$ are $2\pi$-periodic functions, strictly concave on $(0, 2\pi)$, and either all are continuously differentiable on $(0, 2\pi)$ or for each $j=0, \ldots ,n$ we have
\[
\lim \limits_{t \uparrow 2\pi} D_+ K_j(t) = 
\lim \limits_{t \uparrow 2\pi} D_- K_j(t) = -\infty, 
\quad \text{or} \quad 
\lim \limits_{t \downarrow 0} D_- K_j(t) = 
\lim \limits_{t \downarrow 0} D_+ K_j(t) = -\infty.
\]
Here $D_{\pm} K_j$ denote the (everywhere existing) one sided derivatives of the function $K_j$. 

In the present paper we prove an analog of Theorem~\ref{theorem:FNR} on $\RR,$ by reducing minimax and maximin problems on the axis to the case of a segment.

\subsection{Formulation of the problems on the real axis}

Among the generalizations of Chebyshev's alternation theorem, there exists a fairly general result for the so-called $T$-systems. However, this approach does not work for Bojanov's case. On the other hand, the sums of translates approach successfully addresses it. Let us discuss this in detail.

Recall that a $T$-system on a set $A$ is a set of continuous real-valued functions $\{u_j(x)\}_{j=0}^n$ defined on $A$ with the property that all the nontrivial generalized polynomials of the form $\ds \sum \limits_{j=0}^n a_j u_j(x)$ have at most $n$ zeros on $A$.

In 1926, S.~N.~Bernstein \cite{Bernstein} obtained the following generalization of Chebyshev's alternation theorem for $T$-systems on the segment $[a,b]$.
\begin{NotOurTheorem} \label{theorem:TSystem}
If $\{u_j(x)\}_{j=0}^n$ is a $T$-system on $[a,b]$, then for any continuous function $f$ on $[a,b]$ there exists a unique polynomial $u_*$ that minimizes the problem of finding
\[
\min_{a_0,\ldots,a_n} \max_{x \in [a,b]} \left| f(x) - \sum \limits_{j=0}^n a_j u_j(x) \right|.
\]
This polynomial is characterized by the existence of $n+2$ points $a \le x_1 < \ldots < x_{n+2} \le b$ such that
\[
(-1)^j \delta (f(x_j)-u_*(x_j)) =
\max \limits_{x \in [a,b]} |f(x)-u_*(x)|, \quad j = 1,\ldots,n+2, \quad \delta = \pm 1.
\]
\end{NotOurTheorem}

Let $\{u_j(t)\}_{j=0}^n$ be a $T$-system on $\RR$, let $w: \RR \to (0, \infty)$ and $f: \RR \to \RR$ be continuous functions. Assume that
\begin{gather*}
\lim \limits_{t \to \pm \infty} w(t)u_j(t) = 0, \ j=0,\ldots, n, \quad
\lim \limits_{t \to \pm \infty} w(t) f(t) = 0.
\end{gather*}
It is easy to see that $\{w(t)u_j(t)\}_{j=0}^n$ is a $T$-system on $\RR$. Note that if we set $w(-\pi/2)u_j(-\pi/2)=w(\pi/2)u_j(\pi/2) = 0,$ then $\{w(\tan x)u_j(\tan x)\}_{j=0}^n$ is a $T$-system on $[-\pi/2,\pi/2]$. This follows from the fact that the function $\tan$ strictly increases on $(-\pi/2,\pi/2)$ and maps this interval onto $\RR$. Applying Theorem \ref{theorem:TSystem} to the system $\{w(\tan x)u_j(\tan x)\}_{j=0}^n$ with the approximated function $w(\tan x) f(\tan x)$, and then making the substitution $\tan x=t$, we conclude that in the case of $T$-systems, Bernstein's theorem for the weighted minimax problem on $\RR$ is valid. For $f(t) = t^{n+1}$, $\{u_j(t)\}_{j=0}^n = \{t^j\}_{j=0}^n$ we obtain the problem for weighted algebraic polynomials of degree $n+1$.

However, attempts to apply that approach to our problem directly fail quickly because even for the Bojanov problem, the occurring polynomials do not form a vector space, as linear combinations of them can have different root multiplicities. 

The sum of translates approach allows one to obtain a result for Bojanov's case as well. We present a characterization of the extremal polynomial in the weighted Bojanov problem in Corollary \ref{cor:Bojanov}. Moreover, as in Theorem \ref{theorem:FNR}, we deal with almost arbitrary weights.

To formulate our problems precisely, we need some additional preparation.

Let $K: (-\infty, 0) \cup (0, \infty) \to \RR$ be a kernel function and $J: \RR \to \uRR$ be 
 an $n$-field function. 
Let us introduce the set
\[
\Sinf := \{\yy = (y_1, \ldots, y_n) \colon -\infty < y_1 \le y_2 \le \ldots \le y_n < \infty\}.
\]
Let $r_1,\ldots ,r_n > 0.$
Consider the sum of translates function
\[
F(\yy, t) = J(t) + \sum \limits_{j=1}^n r_j K(t-y_j), \quad \yy \in \Sinf, \quad t\in \RR.
\]

\begin{Def}
Let $R >0$ and let $K$ be a kernel function defined on $\RR$. A field function $J$ defined on $\RR$ is said to be $R$-admissible (for $K$) if 
\[
\lim_{|t| \to \infty} \left( J(t) + R K(t) \right) = - \infty.
\]
\end{Def}

\begin{Rem}
Note that if $J$ is $R$-admissible for $K$, then it is also $\widetilde{R}$-admissible for any $\widetilde{R} \in (0, R)$. Without loss of generality, consider positive $t.$ If $C := \lim \limits_{t \to \infty} K(t) \in \RR$, then $\widetilde{R} K(t) \le R (K(t) - C + 1)$ for large $t$, and hence $J$ is $\widetilde{R}$-admissible. Otherwise, due to the concavity, $\lim \limits_{t \to \infty} K(t) = -\infty$. Since $J$ is bounded above by definition, it follows that $J$ is $\widetilde{R}$-admissible.
\end{Rem}

In what follows, we will consider only admissible fields. If we do not specify anything about $R$, it means that we are considering sums of translates with multiplicities $r_1, \ldots ,r_n$, and
\[
R := r_1 + \ldots + r_n.
\]
Denote
\begin{align*}
m_0(\yy) & := \sup \limits_{t \in (-\infty, y_1]} F(\yy, t), \quad
m_n(\yy) := \sup \limits_{t \in [y_n, \infty)} F(\yy, t),\\
m_j(\yy) & := \sup \limits_{t \in [y_j, y_{j+1}]} F(\yy, t), \quad j = 1, \ldots, n-1.
\end{align*}
Similarly to the case of $[a,b]$, we say that $\yy \in \Sinf$ is an equioscillation point on $\mathbb{R}$ if
\[
m_0(\yy) = m_1(\yy) = \ldots = m_n(\yy).
\]
We study the quantities
\begin{gather*}
\mol(\yy) := \max \limits_{j=0,\ldots,n} m_j(\yy) = \sup \limits_{t \in \RR} F(\yy,t), \quad \mul(\yy) := \min \limits_{j=0,\ldots,n} m_j(\yy), \\
M(\Sinf) := \inf \limits_{\yy \in \Sinf} \mol(\yy), \quad m(\Sinf) := \sup \limits_{\yy \in \Sinf} \mul(\yy).
\end{gather*}
It is easy to see that $\mol(\yy)$ is finite for all $\yy \in \Sinf$. 
Indeed, by \eqref{def:admissible}, $F(\yy,\cdot)$ is bounded above outside some segment. Since $J$ and $K$ are bounded above on this segment, $F(\yy,\cdot)$ is bounded above too. Therefore, $F(\yy,\cdot)$ is bounded above on $\RR$, so $\mol(\yy) < \infty$. On the other hand, $\mol(\yy) > -\infty,$ since $F(\yy,\cdot)$ is finite at least at one point by the definition of $J$.

Our goal is to prove the following theorem.
\begin{Teo} \label{theorem:main}
Let $K: (-\infty, 0) \cup (0, \infty) \to \RR$ be a monotone kernel function. Let $n \in \NN, \ r_1, \ldots, r_n > 0$ be arbitrary, and define $R := r_1 + \ldots + r_n$. Assume that $J: \RR \to \uRR$ is an $R$-admissible $n$-field function for $K$. Consider the sum of translates function \eqref{eq:sumOfTranslates}.  Then 
\[
M(\Sinf) = m(\Sinf)
\]
and 
there exists some point $\ww \in \Sinf$ such that
\[
\mol(\ww)=M(\Sinf).
\]
Furthermore, there are no $\xx, \yy$ with finite local maxima $m_j$ such that
\begin{gather} \label{theorem:main:weakIntertwining}
m_j(\xx) > m_j(\yy), \quad j = 0,\ldots,n. 
\end{gather}
If there exists an equioscillation point $\ee$ on $\RR$, then it is the minimax and maximin point, i.e.,
\[
M(\Sinf) = \mol(\ee) = \mul(\ee) = m(\Sinf).
\]
Additionally, 
\begin{enumerate}
\item If $J$ is upper semicontinuous or $K$ is singular, then there exists an equioscillation point on $\RR$.
\item If $K$ is singular and strictly concave, then the equioscillation point is unique. Moreover, the intertwining property holds: the strict inequality \eqref{theorem:main:weakIntertwining} can be replaced by the non-strict one, if $\xx \neq \yy.$ This in particular implies that the equioscillation point is the unique minimax and maximin point.
\end{enumerate}
\end{Teo}

\section{Simple lemmas}
Let us formulate an auxiliary assertion, an equivalent description of concavity, which will be very useful for us below. This statement is well-known, its proof can be found e.g. in \cite[Lemma 10]{Rankin}, but we still give it in a more convenient formulation for us.
\begin{Lem} \label{lemma:concaveShift}
Let $g$ be a concave function on a segment. If $x$ belongs to this segment and $k, h > 0$ are such that $x + k + h$ also belongs to it, then
\[
g(x+k+h)-g(x+h) \le g(x+k)-g(x).
\]
\end{Lem}
\begin{proof}
By definition of concavity, for the point $x+k$ we have
\[
\dfrac{k}{k+h} g(x+k+h) + \dfrac{h}{k+h} g(x) \le g(x+k),
\]
and for $x+h$
\[
\dfrac{h}{k+h} g(x+k+h) + \dfrac{k}{k+h} g(x) \le g(x+h).
\]
Summing these inequalities, we obtain
\[
g(x+k+h)+g(x) \le g(x+k)+g(x+h),
\]
and the lemma is proved.
\end{proof}

\begin{Lem}
If $J$ is admissible, then for any $\yy \in \Sinf$ 
\begin{gather} \label{def:admissible}
\lim \limits_{|t| \to \infty} F(\yy,t) = -\infty.
\end{gather}
\end{Lem}
\begin{proof}
Let us prove the lemma for $t \to \infty$. If $t \to -\infty$, then we can consider $\widetilde{J}(t) = J(-t)$ and $\widetilde{K}(t) = K(-t)$ and apply what is proved for $t \to \infty$.

Firstly, let us show that for any $y \neq 0$ 
\[
\lim \limits_{t \to \infty} J(t) + RK(t-y) = -\infty.
\]

We consider two cases. 
First, let $K$ be a kernel which is
monotonically increasing on $(0,\infty)$.
We have
\[
J(t) + R K(t-y) = J(t) + R K(t) + R (K(t-y)-K(t)).
\]
If $K(t-y)-K(t) \le 0,$ then, obviously, since $J$ is admissible, we obtain our statement. Now let $K(t-y)-K(t) > 0.$ Hence, by monotonicity of $K$, for large $t$ we have that $t-y > t > 1.$ Here applying Lemma \ref{lemma:concaveShift} with $x = 1, \ k = -y, \ h = t-1,$ we obtain that
\[
K(t-y)-K(t) \le K(1-y) - K(1).
\]
This estimate does not depend on $t.$ Hence, using the admissibility of $J$, we have that
\[
J(t) + R K(t-y) \le J(t) + R K(t) + R (K(1-y) - K(1))\to -\infty, \quad t \to \infty.
\]

Now suppose that $K$ is not monotone on $(0, \infty)$. Therefore, $\lim \limits_{t \to \infty} K(t) = -\infty$. By definition of a field, $J$ is bounded above, and we have that $J(t) + R K (t-y) \to -\infty$ as $t \to \infty$.

Finally, we have that
\[
F(\yy,t) = J(t) + \sum \limits_{j=1}^n r_j K(t-y_j) = \sum \limits_{j=1}^n \dfrac{r_j}{R} \left( J(t) + R K(t-y_j) \right) \to -\infty, \quad t \to \infty.
\]
\end{proof}

\section{Mhaskar-Rakhmanov-Saff theorem for sums of translates \\ on the real axis}
The key in our proof is a reduction to the problem on a segment. For this, we work out an analogue of the Mhaskar-Rakhmanov-Saff theorem.
In 1985, H.~N.~Mhaskar and E.~B.~Saff proved the following theorem \cite{MRS}.

\begin{NotOurTheorem}
Let $w: \mathbb{R} \to [0,\infty)$ be a function with support $\Sigma$ such that
\begin{enumerate}
\item $\Sigma$ has positive logarithmic capacity.
\item The restriction of $w$ to $\Sigma$ is continuous on $\Sigma.$
\item The set $Z := \{x \in \Sigma:\ w(x) = 0\}$ has logarithmic capacity zero. 
\item If $\Sigma$ is unbounded, then $|x| w(x) \to 0$ as $|x| \to \infty, \ x \in \Sigma.$
\end{enumerate}
Then there exists a compact set $S_w$ such that for any polynomial $p$ of degree $n$ if
\[
|w^n(x)p(x)| \le M
\]
holds quasi-everywhere on $S_w$, i.e., except for a set of logarithmic capacity zero, then it holds quasi-everywhere on $\Sigma$. 
\end{NotOurTheorem}
The key to proving this theorem is to solve the weighted potential theory problem of minimizing logarithmic energy. The support of the extremal measure for this problem is $S_w$. The minimization of logarithmic energy has physical significance and is related to the potential energy of charged particles. The study of this problem goes back
to Gauss. In~1935, O.~Frostman studied the weighted problem, assuming that $w$ is continuous and superharmonic ~\cite{Frostman}. Mhaskar and Saff adapted Frostman's approach in their proof. In addition to the original work \cite{MRS}, more details can be found in \cite[p.~153]{Saff-Totik}.

E.~A.~Rakhmanov independently explored the connection between potential theory and the minimax problem in 1982 \cite{Rakhmanov}. Rakhmanov was interested in the asymptotic properties of orthogonal polynomials with Hermite and Laguerre weights on the real axis, including their logarithmic asymptotics and distribution of zeros. In addition, Rakhmanov investigated the continuous analogue of the minimax problem for polynomials, where logarithmic potentials are considered instead of polynomials. He showed that such a problem has a unique solution. Throughout his work, there is a clear connection between polynomials and potentials. Rakhmanov is always mentioned alongside Mhaskar and Saff because their ideas overlap.

It is natural to ask about finding $S_w$. This question has been extensively studied when $w$ is an even differentiable function and $x \cdot \dfrac{w'(x)}{w(x)}$ is negative and decreasing on $(0, \infty).$ In this case, $S_w$ is a symmetric segment, its upper bound is called the Mhaskar-Rakhmanov-Saff number and can be found as the solution of a certain integral equation~\cite[p.~216]{Saff-Totik}.

The minimal set on which the absolute values of any weighted polynomial attain their maxima is called the minimal essential set. It is important to note that the set $S_w$ from the Mhaskar-Rakhmanov-Saff theorem is also the minimal essential set \cite{Gonchar}.

In our paper, we do not use the approach of potential theory (and the concept of logarithmic capacity, respectively), and our inequalities hold everywhere. 

The following result, essential to our argument in proving Theorem \ref{theorem:main}, was communicated to us by Szil\'{a}rd Gy. R\'{e}v\'{e}sz. We present it here with his permission.

\begin{Teo} (Sz.~Gy.~R\'{e}v\'{e}sz) \label{theorem:MRS}
Assume that $K$ is a monotone kernel, $R>0, \ n \in \NN$ and $J$ is an $R$-admissible $n$-field for $K.$ Then there exists $q=q(K,J,R)$ such that for any positive $r_1, \ldots, r_n$ with $\ds \sum \limits_{j=1}^n r_j \le R$ and for~all $\yy \in \Sinf$ we have
\[
\mol(\yy) = \sup \limits_{t \in [-q,q]} F(\yy, t).
\]

\end{Teo}

\begin{proof}
Assume that the set of points where $J$ is finite is unbounded. Otherwise, the statement is trivial.

Fix $n+1$ arbitrary points $z_0 < z_1 < \ldots < z_n$ such that $J(z_j) > -\infty$ and $z_{j+1}-z_j > 2, \ j=0,\ldots,n$.
Take $\yy \in \Sinf$. There are $n+1$ non-intersecting $1$-neighborhoods of $z_j, \ j=0,\ldots,n$. Take $\yy \in \Sinf$. By Dirichlet's principle, there is $z_i$ such that $\min \limits_{j=1,\ldots,n} |y_j-z_i| > 1.$ Set $z := z_i$.

Let us estimate
\[
F(\yy,t) = J(t) + \sum \limits_{y_j < z} r_j K(t-y_j) + \sum \limits_{y_j > z} r_j K(t-y_j)
\]
for large $t > z+1.$

If $y_j > t > z$, since $K$ is monotone, we have 
\[
K(t-y_j) - K(z-y_j) \le 0 \le K(t-z+1) - K(1).
\]
If $t > y_j > z,$ also using monotonicity, we get
\[
K(-1) - K(z-y_j) \le 0 \le K(t-z+1) - K(t-y_j).
\]
If $y_j < z$, then, by the choice of $z$, we have $y_j < z-1$. Applying Lemma \ref{lemma:concaveShift} with $x = 1, \ k = z-y_j-1, \ h = t-z$, we obtain
\[
K(t-y_j) - K(z-y_j) \le K(t-z+1) - K(1).
\]
In all three cases we get
\[
K(t-y_j) - K(z-y_j) \le 
K(t-z+1) - \min\{K(-1),K(1)\}.
\]
Using this and monotonicity of $K$, we obtain
\begin{align*}
F(\yy,t) &\le F(\yy, z) + J(t) + \sum \limits_{j=1}^n r_j (K(t-z+1) - \min\{K(-1),K(1)\}) - J(z) \\
& \le F(\yy, z) + J(t) + R(K(t-z+1) - \min\{K(-1),K(1)\}) - J(z).
\end{align*}
Since $J$ is $R$-admissible, for any $x \in \RR$ there exists $\widetilde{q}(x) = \widetilde{q}(x, K, J, R) \ge |x|$ such that
\[
J(t) + R (K(t-x+1) - \min\{K(-1),K(1)\}) - J(x) \le 0, \quad t > \widetilde{q}(x).
\]
Note that $z \in \{z_0,\ldots,z_n\}$ for any $\yy \in \Sinf$. Set $q := \max \limits_{j=0,\ldots,n} \widetilde{q}(z_j).$ Therefore, for any $\yy \in \Sinf$ we obtain
\[
F(\yy,t) \le F(\yy, z), \quad t > q.
\]

The proof for negative $t$ is carried out in a similar way.
\end{proof}

\begin{Def}
Let $K$ be a monotone kernel, $R>0, \ n \in \NN$ and $J$ be an $R$-admissible $n$-field for $K.$ Consider a number $q$ such that $\mol(\yy)=\sup \limits_{t \in [-q,q]} F(\yy, t)$ for all $r_1, \ldots, r_n > 0$ with $\ds \sum \limits_{j=1}^n r_j \le R$ and for all $\yy \in \Sinf$. Then we say that $q$ has the Mhaskar-Rakhmanov-Saff property for the system $K, J$~and~$R.$
\end{Def}

In what follows, we consider $r_1, \ldots, r_n > 0$ and $\ds \sum \limits_{j=1}^n r_j = R$.

The following lemmas on numbers with Mhaskar-Rakhmanov-Saff property are necessary to prove the main result.

\begin{Lem} \label{lemma:equioscillation}
Let $K$ be a monotone kernel, $J$ be an admissible field and let $q$ have the Mhaskar-Rakhmanov-Saff property. Assume that $\ww \in \Simplex{\ell}$ is an equioscillation point on $[-\ell,\ell],$ where $\ell \ge q.$ Then $\ww$ is an equioscillation point on $\RR.$ 
\end{Lem}

\begin{proof}
We have 
\begin{align}
m_0(\ww) & \ge m^{\ell}_0(\ww), 
\quad m_n(\ww) \ge m^{\ell}_n(\ww), \label{lemma:equioscillation:eq1} \\
m_j(\ww) & = m^{\ell}_j(\ww), \ j = 1,\ldots,n-1. \label{lemma:equioscillation:eq2}
\end{align}
On the other hand, by definition of $q$ and the equioscillation property of $\ww$ on $[-\ell,\ell] \supseteq [-q,q],$
\[
\mol(\ww) = \molsegm{\ell}(\ww) = \max \limits_{k=0,\ldots,n} m^{\ell}_k(\ww) = m^{\ell}_j(\ww), \quad j=0,\ldots,n.
\]
Thus we have
\begin{gather} \label{lemma:equioscillation:eq3}
m_0(\ww) \le \mol(\ww) = m^{\ell}_0(\ww), 
\quad m_n(\ww) \le \mol(\ww) = m^{\ell}_n(\ww).
\end{gather}
By \eqref{lemma:equioscillation:eq1}, \eqref{lemma:equioscillation:eq2} and \eqref{lemma:equioscillation:eq3}, we have for all $j \in \{0,\ldots,n\}$ 
\[
m_j(\ww) = m^{\ell}_j(\ww).
\]
The statement is proved.
\end{proof}

\begin{Lem} \label{lemma:minimax:compactRoot}
Let $K$ be a monotone kernel, $J$ be an admissible field function and $q$ have the Mhaskar-Rakhmanov-Saff property. If $\ell \ge q$, then for any $\yy \in \Sinf \setminus \Simplex{\ell}$ we have
\[
\mol(\yy') \le \mol(\yy),
\]
where
\[\yy'=(y'_1,\ldots,y'_n), \quad
y'_j = 
\begin{cases} -\ell, & y_j < -\ell,\\
y_j, & y_j \in [-\ell,\ell],\\
\ell, & y_j > \ell.
\end{cases}
\]
\end{Lem}

\begin{proof} By monotonicity of $K,$ for $t \in [-\ell,\ell]$
\begin{align*}
F(\yy,t) & = J(t) + \sum \limits_{y_j < -\ell} r_j K(t-y_j) + \sum \limits_{y_j \in [-\ell,\ell]} r_j K(t-y_j) + \sum \limits_{y_j > \ell} r_j K(t-y_j) \\
& \ge J(t) + \sum \limits_{y_j < -\ell} r_j K(t+\ell) + \sum \limits_{y_j \in [-\ell,\ell]} r_j K(t-y_j) + \sum \limits_{y_j > \ell} r_j K(t-\ell) = F(\yy',t).
\end{align*}
Hence 
\[
\sup \limits_{t \in [-\ell,\ell]} F(\yy, t) \ge \sup \limits_{t \in [-\ell,\ell]} F(\yy', t).
\]
Finally, note that by definition of $q$ and because of $q \le \ell$, we have $\sup \limits_{t \in [-\ell,\ell]} F(\yy, t) = \mol(\yy)$ and $\sup \limits_{t \in [-\ell,\ell]} F(\yy', t) = \mol(\yy').$ The lemma is proved.
\end{proof}

\section{Other lemmas}
The following statement is known, but we present its proof.
\begin{Lem} \label{lemma:uniformContin}
Suppose that a function $g$ is concave on the semiaxis $[M,\infty)$ (on $(-\infty,-M]$), is nondecreasing on $[M,\infty)$ (nonincreasing on $(-\infty,-M]$) and is continuous at $M$ (at $-M$). Then $g$ is uniformly continuous on $[M,\infty)$ (on $(-\infty,-M]$).
\end{Lem}

\begin{proof}
Now let us prove the statement for $[M, \infty)$, i.e., that
\[
\forall \ \varepsilon > 0 \ \exists \ \delta(\varepsilon) > 0 \ \forall \ t_1, t_2 \in [M,\infty) \quad |t_1-t_2| < \delta(\varepsilon) \implies |g(t_1)-g(t_2)| < \varepsilon.
\]
Since $g$ is continuous at $M,$
\begin{gather} \label{lemma:uniformContin:eq1}
\forall \ \varepsilon > 0 \ \exists \ \delta(\varepsilon) > 0 \quad 0 < t - M < \delta(\varepsilon) \implies |g(t)-g(M)| < \varepsilon.
\end{gather}
Fix an arbitrary $\varepsilon > 0$ and $\delta(\varepsilon)$ from \eqref{lemma:uniformContin:eq1}. Take $t_1, t_2 \in [M, \infty)$ and assume that $t_2 = t_1 + h, \ 0 \le h < \delta(\varepsilon).$ By Lemma \ref{lemma:concaveShift} and monotonicity of $g$, we get
\[
0 \le g(t_2)-g(t_1) \le g(M+h)-g(M).
\]
Taking into account \eqref{lemma:uniformContin:eq1}, we have obtained the uniform continuity of $g$ on $[M, \infty).$

The proof for $(-\infty, -M]$ is similar.
\end{proof}

\begin{Lem} \label{lemma:-infty}
If $K \colon (-\infty, 0) \cup (0, \infty) \to \RR$ is a kernel function and $J \colon \RR \to \uRR$ is an admissible field function, then for any $\yy \in \Sinf$
we have
\[
\lim \limits_{\xx \to \yy, \ |t| \to \infty}
F(\xx,t) = -\infty.
\]
\end{Lem}
\begin{proof}
For brevity, let us prove the lemma for $t \to \infty.$ For $t \to -\infty$ the proof is carried out analogously.

1. Suppose that $K$ is nondecreasing on $(0, \infty)$. By Lemma \ref{lemma:uniformContin}, the function $K$ is uniformly continuous on $[1,\infty)$. Hence if $\|\xx - \yy\|$ is sufficiently small, then for large $t$ the differences $K(t-x_j) - K(t-y_j)$ are bounded above. 

Moreover, since $J$ is admissible, 
\[
\lim \limits_{|t| \to \infty}
F(\yy,t) = -\infty.
\]
Thus we have
\[
F(\xx,t) = F(\yy,t) + \sum \limits_{j=1}^n r_j (K(t-x_j) - K(t-y_j)) \to -\infty, 
\quad \xx \to \yy, \ t \to \infty.
\]

2. Assume that $K$ is not nondecreasing on $(0, \infty)$. Hence, by concavity, $K$ decreases for large $t$. As $\xx \to \yy,$ one can assume that $x_j \le y_j + 1$ for all $j.$ So, for large $t$ we have that $K(t-x_j) \le K(t-y_j-1)$ for all $j$. By the admissibility of $J$, we have
\[
J(t) + \sum \limits_{j=1}^n r_j K(t-x_j) \le J(t) + \sum \limits_{j=1}^n r_j K(t-y_j-1) \to -\infty.
\]
Hence $\lim \limits_{\xx \to \yy, \ t \to \infty} F(\xx,t) = -\infty.$ 
\end{proof}

\begin{Lem} \label{lemma:M_L} Let $K \colon (-\infty, 0) \cup (0, \infty) \to \RR$ be a kernel and $J: \RR \to \uRR$ be an admissible field. 
Then for any $L > 0$ there is $M_L \ge L$ such that for each $ \yy \in \Simplex{L}$
\begin{gather*}
m^{M_L}_j(\yy) = m_j(\yy), \quad j = 0, \ldots, n, \\
\mol(\yy) = \molsegm{M_L}(\yy), \quad \mul(\yy)=\mulsegm{M_L}(\yy).
\end{gather*}
\end{Lem}

\begin{proof}
It is enough to prove that for any $L > 0$ there is $M_L \ge L$ such that for any $\yy \in \Simplex{L}$ and $t \in \RR$
\[
\left( t < -M_L \implies F(\yy,t) \le \sup \limits_{-M_L \le t \le -L} F(\yy,t) \right)
\]
and 
\[
\left( t > M_L \implies F(\yy,t) \le \sup \limits_{L \le t \le M_L} F(\yy,t) \right).
\]
Let us prove the statement for $t < -M_L$. The proof of the other one is similar.

Assume for a contradiction that for some $L$
\begin{equation}
\label{lemma:M_L:assumption}
\begin{gathered}
\forall N \in \NN, \ N \ge L \quad \exists \,\yy_N \in \overline{S^L} \quad \exists\, t_N \in \RR \\ \left(t_N < -N \quad \text{and} \quad F(\yy_N,t_N) > \sup \limits_{-N \le t \le -L} F(\yy_N,t) \right).
\end{gathered}
\end{equation}
We have the bounded sequence $\{\yy_N\} \subset \Simplex{L}.$ By the Bolzano-Weierstrass theorem, there is a convergent subsequence $\{\yy_{N_k}\}.$
Set
\[
\lim \limits_{N_k \to \infty} \yy_{N_k} = \yy^{*} = (y^*_{1},\ldots,y^*_{n}) \in \Simplex{L}.
\]
Take $u~\in~(-\infty, -L)$. By our indirect assumption \eqref{lemma:M_L:assumption}, we have for all $k$ with $-N_k < u < -L$
\[
J(u) = F(\yy_{N_k},u) - \sum \limits_{j=1}^n r_j K(u - y^{N_k}_j) \le F(\yy_{N_k},t_{N_k}) - \sum \limits_{j=1}^n r_j K(u - y^{N_k}_j).
\]
Using continuity of $K$ at $u - y^*_j < 0$ and Lemma \ref{lemma:-infty}, we get, since $t_{N_k} \to -\infty$, that
\begin{align*}
J(u) & \le \lim \limits_{k \to \infty} \left( F(\yy_{N_k},t_{N_k}) - \sum \limits_{j=1}^n r_j K(u - y^{N_k}_j) \right) \\
& = \lim \limits_{k \to \infty} F(\yy_{N_k},t_{N_k}) - \sum \limits_{j=1}^n r_j K(u - y^*_j) = -\infty.
\end{align*}
So, $J(u) \equiv -\infty$ for $u < -L.$ We have a contradiction with our assumption, because then also $F(\yy_N, t_N) = -\infty.$
\end{proof}

In our proof we need Jensen's inequality \cite[p.~32]{Jensen}. Let $r_1, \ldots, r_n~>~0$, $g$ be a concave function on a segment $I$ and $y_1, \ldots, y_n \in I$. Denote by $c(\yy)$ the weighted average 
\[
c(\yy) := \dfrac{r_1 y_1 + \ldots + r_n y_n}{r_1 + \ldots + r_n}.
\]
Then the following inequality holds
\begin{gather} \label{eq:Jensen}
\sum \limits_{j=1}^n r_j g(y_j) \le \sum \limits_{j=1}^n r_j g(c(\yy)). 
\end{gather}

\begin{Lem} \label{lemma:maximin:compactRoot}
Let $K$ be a monotone kernel and $J$ be an admissible field. Then there exists $L$ such that
\begin{gather}\label{lemma:maximin:compactRoot:result}
m(\Sinf)=m(\Simplex{L}).
\end{gather}
\end{Lem}

\begin{proof}
Fix $\xx := (x_1,\ldots,x_n) \in \Sinf$ such that $\mul(\xx) > -\infty.$ Such $\xx$ exists, since one can take nodes of $\xx$ strictly between $n+1$ arbitrary points where $J$ is finite. 

Denote $\mathbf{0} := (0,\ldots,0).$ Since $J$ is admissible, there exist $L \ge L_1 \ge \max \{|x_1|, |x_n|\}$ such that 
\begin{gather}
\label{lemma:maximin:compactRoot:L_1} 
F(\mathbf{0},t) \le \mul(\xx), \quad |t| > L_1, \\
\label{lemma:maximin:compactRoot:L_2} 
J(t) + \sum \limits_{j=1}^n r_j \max \{K(t-L_1), K(t+L_1)\} \le \mul(\xx), \quad |t| > L.
\end{gather}
Consider $\yy \notin \Simplex{L}.$

Let $\cc := (c(\yy), \ldots, c(\yy)).$ 
By Jensen's inequality \eqref{eq:Jensen},
\begin{gather} \label{lemma:maximin:compactRoot:Jensen}
F(\yy,t) \le F(\cc,t), \quad t \in (-\infty, y_1) \cup (y_n, \infty).
\end{gather}

First, assume that $|c(\yy)| > L_1.$ If $c(\yy) > L_1,$ then it is necessary that $y_n > L_1,$ and we consider $t > y_n.$ If $c(\yy) < -L_1,$ then we take $t < y_1 < -L_1.$  Using successively \eqref{lemma:maximin:compactRoot:Jensen}, monotonicity of $K$ and \eqref{lemma:maximin:compactRoot:L_1}, we have for these $t$
\begin{gather}
\label{lemma:maximin:compactRoot:case_1}
F(\yy,t) \le F(\cc,t) \le F(\mathbf{0},t) \le \mul(\xx).
\end{gather}

Now suppose $|c(\yy)| \le L_1.$ Since $\yy \notin \Simplex{L}$, either $y_1 < -L,$ or $y_n > L$.
If $y_1 < -L$,  we consider $t < y_1$. Otherwise $y_n > L$, and we take $t > y_n.$
Using~\eqref{lemma:maximin:compactRoot:Jensen} and the monotonicity of $K$ and \eqref{lemma:maximin:compactRoot:L_2}, we get for these $t$
\begin{gather}
\label{lemma:maximin:compactRoot:case_2}
F(\yy,t) \le F(\cc,t) \le J(t) + \sum \limits_{j=1}^n r_j \max \{K(t-L_1), K(t+L_1)\} \le \mul(\xx).
\end{gather}

By \eqref{lemma:maximin:compactRoot:case_1} and \eqref{lemma:maximin:compactRoot:case_2}, we conclude that in both cases
\[
\mul(\yy) \le \mul(\xx), \quad \yy \notin \Simplex{L}.
\]
Since $\xx \in \Simplex{L},$ it follows that 
\[
m(\Sinf) = m(\Simplex{L}).
\]
\end{proof}

\section{Proof of the main result}
\subsection{Reduction of the problem on $[-\ell,\ell]$ to the problem on $[0,1]$} \label{subsection:reductionToSegment}
Suppose that $K$ is a kernel function defined on $[-2\ell,2\ell]$ and $J$ is an $n$-field function on $[-\ell,\ell].$ Let us reduce the minimax and maximin problems on $[-\ell,\ell]$ to the segment $[0,1]$.
Let
\[
t \in [-\ell,\ell], \quad x = x(t) := \dfrac{t}{2\ell}+\dfrac{1}{2} \in [0,1].
\]
Consider
\begin{align*}
K^1(x) & := K(2 \ell x), \quad x \in [-1,1], \\
J^1(x) & := J(2\ell(x-1/2)), \quad
x \in [0,1].
\end{align*}
Note that if $K$ is a singular (strictly) concave and (strictly) monotone kernel function and $J$ is an $n$-field function, then $K^1$ and $J^1$ have the same properties. 
We have
\begin{gather*}
K(t-y_j)=K^1(x(t)-x(y_j)), \quad y_j \in [-\ell, \ell], \ j=1,\ldots,n, \\
J(t)=J^1(x(t)).
\end{gather*}
Consider
\[
F(\yy, t) = J(t) + \sum \limits_{j=1}^n r_j K(t-y_j), \quad t \in [-\ell,\ell], \ \yy \in \Simplex{\ell}.
\]
It equals to
\[
F^1(\yy^1,x(t)) := J^1(x(t)) + \sum \limits_{j=1}^n r_j K^1(x(t)-x(y_j)), \quad \yy^1 := (x(y_1),\ldots,x(y_n)).
\]
In short,
\begin{gather} \label{reductToSegment:identity}
F(\yy,t) = F^1(\yy^1, x), \quad t \in [-\ell,\ell], \ x = x(t) = \dfrac{t}{2\ell} + \dfrac{1}{2} \in [0,1].
\end{gather}
From this identity it is clear that all results for $[0,1]$ mentioned in the introduction are also valid for $[-\ell,\ell].$

\subsection{Proof of Theorem \ref{theorem:main}}
\begin{proof}
1. Firstly, let us prove that there is a minimax point. Let $q$ be a number with the Mhaskar-Rakhmanov-Saff property. Take some $\ell > q$ such that $J$ is finite at least at $n+1$ points on $(-\ell,\ell).$ 
\footnote{When $K$ is singular, then $q$ can have the Mhaskar-Rakhmanov-Saff property only if there are at least $n+1$ finiteness points of $J$ on $[-q,q]$, hence any $\ell > q$ works here, but this is not necessarily true if $K$ is non-singular.} 
Therefore, it makes sense to consider the minimax problem on $[-\ell,\ell]$. As is discussed in Subsection~\ref{subsection:reductionToSegment}, it reduces to the problem on $[0,1].$
By Theorem \ref{theorem:FNR}, there exists a minimax point $\ww \in \Simplex{\ell}$ for the problem on $[-\ell,\ell]$. 

Let us show that $\ww$ is the extremal point for the minimax problem on $\RR$, too.

If $\yy \in \Simplex{\ell},$ then
\begin{gather} \label{theorem:main:eq1}
\mol(\yy) = \molsegm{\ell}(\yy) \ge \molsegm{\ell}(\ww) = \mol(\ww).
\end{gather}
Here the inequality follows from the fact that $\ww$ is the minimax point for $\molsegm{\ell}$, and the equalities are due to the Mhaskar-Rakhmanov-Saff property of $q$ and our assumption that $\ell > q.$

If $\yy \in \Sinf \setminus \Simplex{\ell}$, then by Lemma \ref{lemma:minimax:compactRoot}, there is $\yy' \in \Simplex{\ell}$ such that $\mol(\yy) \ge \mol(\yy')$. By \eqref{theorem:main:eq1}, $\mol(\yy') \ge \mol(\ww),$ and we obtain
\[
\mol(\yy) \ge \mol(\ww).
\]

Therefore, whether $\yy \in \Simplex{\ell}$ or $\yy \in \Sinf \setminus \Simplex{\ell},$ we have $\mol(\yy) \ge \mol(\ww)$.
So, we have obtained that $\ww$ is a minimax point on $\RR$.

2. Let us show that $M(\Sinf)=m(\Sinf).$ Let $L$ be a number from Lemma~\ref{lemma:maximin:compactRoot}. Without loss of generality, assume that $\ell \ge L.$ Successively using Lemma~\ref{lemma:minimax:compactRoot}, Theorem \ref{theorem:FNR} and Lemma~\ref{lemma:maximin:compactRoot}, we obtain
\[
M(\Sinf)=M(\Simplex{\ell})=m(\Simplex{\ell})=m(\Sinf).
\]

3. Now let us prove that there are no $\xx, \yy$ with finite local maxima $m_j$ such that $m_j(\xx) > m_j(\yy)$ for $j = 0,\ldots,n.$ Take $\xx$ and $\yy$ arbitrarily. Let $C := \max \{ |x_1|, |x_n|, |y_1|, |y_n| \}.$ By Lemma \ref{lemma:M_L}, there is $M_C \ge C$ such that 
\begin{gather*}
m^{M_C}_j(\yy) = m_j(\yy) \text{ and } m^{M_C}_j(\xx) = m_j(\xx), \ j = 0, \ldots, n. 
\end{gather*}
By Theorem \ref{theorem:FNR}, inequalities $m^{M_C}_j(\xx) > m^{M_C}_j(\yy)$ for every $j = 0, \ldots, n$ is impossible. So, we have that inequalities $m_j(\xx) > m_j(\yy), \ j = 0, \ldots, n,$ cannot hold. If $K$ is singular and strictly concave, then, also by Theorem \ref{theorem:FNR}, this inequality can be replaced by the non-strict one, if $\xx \neq \yy.$

4. Note that if $J$ is upper semicontinuous or $K$ is singular, then by Theorem \ref{theorem:FNR}, there is an equioscillation point on $[-\ell,\ell].$ By Lemma \ref{lemma:equioscillation}, it is an equioscillation point on $\mathbb{R}.$ 
If $K$ is singular and strictly concave, as proven above, the intertwining property holds in this case, and this immediately implies the uniqueness of the equioscillation point.

Now, even in the most general case, assuming only that $K$ is monotone and $J$ is admissible, suppose there exists an equioscillation point $\ee$ on $\RR$. Since
\[
M(\Sinf) \le \mol(\ee) = \mul(\ee) \le m(\Sinf) 
\quad \text{and} \quad
M(\Sinf) = m(\Sinf),
\]
equality holds throughout, that is, $\ee$ is an extremal point for both problems. 

The theorem is proved.
\end{proof}

\section{Problems on a semiaxis}
Let $J^+$ be an admissible $n$-field on $[0, \infty)$, $K$ be a monotone kernel defined on $\RR$ and
\[
\Simplex{[0, \infty)} = \{\yy = (y_1, \ldots, y_n): \ 0 \le y_1 \le \ldots \le y_n < \infty \}.
\]
For positive $r_1, \ldots, r_n$ consider the sum of translates on $[0, \infty)$
\[
F^+(\yy,t)= J^+(t) + \sum \limits_{j = 1}^n r_j K(t-y_j), \quad \yy \in \Simplex{[0, \infty)}, \quad t \in [0, \infty).
\]

For $F^+$ with admissible field $J^+$ the minimax and maximin problems on $\Simplex{[0,\infty)}$ also make sense. And it is not hard to see that these problems are equivalent to the problems for 
\[
F(\yy,t)= J(t) + \sum \limits_{j = 1}^n r_j K(t-y_j), \quad \yy \in \Simplex{[0,\infty)}, \quad t \in \RR,
\]
where 
\[
J(t) = 
\begin{cases}
-\infty, & t < 0,\\
J^+(t), & t \ge 0.
\end{cases}
\]
That is, for $F$ we would like to optimize $\mul$ and $\mol$ on $\Simplex{[0,\infty)}$ instead of $\Simplex{\infty}.$ Let us show that all statements of Theorem \ref{theorem:main} are valid for these problems on $[0, \infty).$

\begin{Cor}
Let $n \in \NN, \ r_1, \ldots, r_n > 0$ be arbitrary. Let $K: (-\infty, 0) \cup (0, \infty) \to \RR$ be a monotone kernel function and $J^+$ be an admissible $n$-field function on $[0, \infty)$. Then 
\[
M(\Simplex{[0, \infty)}) = m(\Simplex{[0, \infty)})
\]
and 
there exists some point $\ww \in \Simplex{[0, \infty)}$ such that
\[
\mol(\ww)=M(\Simplex{[0, \infty)}).
\]
Furthermore, there are no $\xx, \yy \in \Simplex{[0, \infty)}$ with finite local maxima $m_j$ such that
\begin{gather} \label{cor:semiaxis:weakIntertwining}
m_j(\xx) > m_j(\yy), \quad j = 0,\ldots,n. 
\end{gather}
If there exists an equioscillation point $\ee$ on $[0, \infty)$, then it is the minimax and maximin point, i.e.,
\[
M(\Sinf) = \mol(\ee) = \mul(\ee) = m(\Sinf).
\]
Additionally, 
\begin{enumerate}
\item If $J^+$ is upper semicontinuous or $K$ is singular, then there exists an equioscillation point on $[0,\infty)$.
\item If $K$ is singular and strictly concave, then the equioscillation point is unique. Moreover, the intertwining property holds: the strict inequality \eqref{cor:semiaxis:weakIntertwining} can be replaced by the non-strict one, if $\xx \neq \yy.$ This in particular implies that the equioscillation point is the unique minimax and maximin point.
\end{enumerate}
\end{Cor}

\begin{proof}
1. Let us show that there is a minimax point $\ww \in \Simplex{[0,\infty)}$. By Theorem~\ref{theorem:main}, there is some minimax point $\ww' = (w'_1,\ldots,w'_n) \in \Sinf$ for the problem on $\RR$ with $F$. Using monotonicity of $K$, it is easy to see that
\[
J(t) + \sum \limits_{j = 1}^n r_j K(t-w'_j) \ge J(t) + \sum \limits_{w'_j < 0} r_j K(t) + \sum \limits_{w'_j \ge 0} r_j K(t-w'_j), \quad t \in [0, \infty).
\]
Hence 
\[
\mol(\ww') \ge \mol(\ww), \text{ where } w_j =
\begin{cases}
0, &w'_j < 0, \\
w'_j, & w'_j \ge 0.
\end{cases}
\]
So, $\ww \in \Simplex{[0,\infty)}$ is also a minimax point. This implies that 
\[
M(\Sinf)=M(\Simplex{[0,\infty)}).
\]

2. Now let us prove that $M(\Simplex{[0,\infty)})=m(\Simplex{[0,\infty)}).$ By Theorem \ref{theorem:main}, $M(\Sinf)=m(\Sinf).$ Note that for all $\yy \in \Sinf \setminus \Simplex{[0,\infty)}$ we have $m_0(\yy)=\mul(\yy) = -\infty$.
Therefore, obviously, $m(\Sinf)=m(\Simplex{[0,\infty)})$.
And we obtain
\begin{gather} \label{eq:semiaxis}
M(\Simplex{[0,\infty)}) = M(\Sinf) = m(\Sinf) = m(\Simplex{[0,\infty)}).
\end{gather}

3. Note that if $\ee$ is an equioscillation point on $\RR$, then $\ee \in \Simplex{[0,\infty)}$. Indeed, as noted above, otherwise we have that $m_0(\ee)=\ldots=m_n(\ee) = -\infty,$ and this is impossible. By Theorem \ref{theorem:main} and \eqref{eq:semiaxis}, 
\[
\mol(\ee) = M(\Simplex{[0,\infty)}) = m(\Simplex{[0,\infty)}) = \mul(\ee).
\]

4. All the other statements are obvious.
\end{proof}

\section{The weighted Bojanov problem on the real axis}
wLet $r_1, \ldots, r_n > 0$ be arbitrary, and denote $\mathbf{r} := (r_1,\ldots,r_n)$. Consider the following sets of monic, generalized, nonnegative polynomials \cite[p.~392]{BorweinErdelyi} of degree $R~:=~\sum \limits_{j=1}^n r_j$  
\begin{gather*}
\mathcal{P}_\rr := \left\{ \prod \limits_{j=1}^n |x-(x_j+iy_j)|^{r_j}: \ -\infty < x_1 < \ldots < x_n < \infty, \ y_1,\ldots,y_n \in \RR \right\},\\
\mathcal{P}_\rr(\RR) := \left\{ \prod \limits_{j=1}^n |x-x_j|^{r_j}: \ -\infty < x_1 < \ldots < x_n < \infty \right\} \subset \mathcal{P}_\rr.
\end{gather*}
Theorem \ref{theorem:main} immediately implies the following version of Theorem \ref{theorem:Bojanov} on $\mathcal{P}_\rr$.
\begin{Cor} \label{cor:Bojanov}
Let $w: \RR \to [0, \infty)$ be a bounded above function assuming non-zero values at more than $n$ points. Assume that
\begin{gather} \label{cor:Bojanov:admissibility}
\lim \limits_{|x| \to \infty} w(x) \cdot x^{R} = 0.
\end{gather}
Then there exists a unique $T \in \mathcal{P}_\rr$
such that
\[
\|w(x)T(x)\| = \inf \limits_{P \in \mathcal{P}_\rr} \|w(x)P(x)\|,
\]
where $\|\cdot\|$ is the sup norm over $\RR$. All roots of $T$ are real and distinct. Moreover, this extremal polynomial is characterized by an equioscillation property: there exists an array of points $-\infty < t_0 < \ldots < t_n < \infty$ such that
\begin{gather} \label{cor:Bojanov:eq}
|T(t_k)|=\|T\|, \quad k=0,\ldots,n.
\end{gather}
\end{Cor}

\begin{proof}
Note that all roots of the extremal polynomials are necessarily real, since if $y_j \neq 0,$ then
\[
|x-(x_j+iy_j)| > |x-x_j|, \quad x \in \RR.
\]
Hence we may consider the minimax problem on $\mathcal{P}_\rr(\RR)$ instead of the original problem.

Let $K := \log |\cdot|, \ J:=\log w.$  It is clear that the original minimax problem is equivalent to the minimax problem for the sums of translates with $K$ and $J.$ Obviously, $K$ is a strictly concave, monotone, singular kernel and $J$ is an admissible $n$-field function for $K$ and for $R.$ Since $K$ is monotone and $J$ is admissible, by Theorem \ref{theorem:main}, there is a minimax point. And because of $K$ is singular and strictly concave, there exists a unique equioscillation point $\xx^* := (x^*_1,\ldots,x^*_n)$ and it is the unique minimax point. Note that from  singularity of $K$ it follows that $x^*_1< \ldots < x^*_n$. And it is clear that the equioscillation property of $\xx^*$ is equivalent to \eqref{cor:Bojanov:eq}.
\end{proof}

The intertwining property for generalized polynomials deserves special attention.
\begin{Cor}
Let $w: \RR \to [0, \infty)$ be a bounded above function assuming non-zero values at more than $n$ points and satisfying \eqref{cor:Bojanov:admissibility}.
Consider $P_1, P_2 \in \mathcal{P}_\rr(\RR)$. Let $x^1_1 < \ldots < x^1_n$ be roots of $P_1$ and $x^2_1 < \ldots < x^2_n$ be roots of $P_2.$ Then the intertwining property holds. Namely, if the inequalities
\begin{gather*}
\max \limits_{x \in (-\infty, x^1_1]} |w(x) P_1(x)| \le \max \limits_{x \in (-\infty, x^2_1]} |w(x) P_2(x)|,\\
\max \limits_{x \in [x^1_j, x^1_{j+1}]} |w(x) P_1(x)| \le \max \limits_{x \in [x^2_j, x^2_{j+1}]} |w(x) P_2(x)|, \quad j=1,\ldots,n-1,\\
\max \limits_{x \in [x^1_n, \infty)} |w(x) P_1(x)| \le \max \limits_{x \in [x^2_n, \infty)} |w(x) P_2(x)|.
\end{gather*}
hold simultaneously, then $P_1=P_2.$
\end{Cor}

\begin{proof}
Consider the sums of translates with $K := \log |\cdot|, \ J:=\log w.$ As mentioned in the proof of Corollary \ref{cor:Bojanov}, $K$ is strictly concave, monotone, singular and $J$ is admissible. By Theorem \ref{theorem:main}, for these sums of translates we have the intertwining property. It remains to note that inequalities between local maxima of two generalized polynomials do not change after taking logarithms.
\end{proof}
Note that the intertwining property on the real axis is new even in classical cases, such as the Chebyshev problem for algebraic polynomials with the Hermite weight $\exp(-x^2)$.

\section*{Acknowledgements}
The author is grateful to Prof. Sz.~Gy.~R\'{e}v\'{e}sz for the problem statement, useful references and discussions, especially for the analogue of the Mhaskar-Rakhmanov-Saff theorem, as well as for the thorough review during the writing of the article.

The author is thankful to P.~Yu. Glazyrina for constant attention to this work, discussions of proofs, advice on the structure of the article.

The author also thanks M.~Pershakov for pointing out useful inequalities for concave functions.

\bigskip
\bigskip
\bigskip
\noindent\parindent0pt
\hspace*{5mm}
\begin{minipage}{\textwidth}
\noindent
\hspace*{-5mm}Tatiana Nikiforova\\
Krasovskii Institute of Mathematics and Mechanics, \\
Ural Branch of the Russian Academy of Sciences\\
620990 Ekaterinburg, Russia,\\
Ural Federal University\\
620002 Ekaterinburg, Russia
\end{minipage}

\end{document}